\documentclass{amsart}
\usepackage{graphicx}
\usepackage{amsfonts}
\usepackage{amssymb}
\usepackage{amsmath}
\usepackage{ulem}
\usepackage[mathscr]{eucal}

\theoremstyle{plain}
\newtheorem{prop}{Proposition}[section]
\newtheorem{coro}[prop]{Corollary}

\newtheorem{thm}[prop]{Theorem}

\theoremstyle{definition}

\newtheorem{exmp}[prop]{Example}

\def\mcg#1;#2{\Gamma_{#1,#2}}
\def\fg#1;#2{\Pi_{#1,#2}}
\def\tb#1;#2{\mathscr{K}_{\frac{#1}{#2}}}

\font\tendb=msbm10 at 12pt \font\sevendb=msbm10 at 9pt
\font\fivedb=msbm10 at 7pt
\newfam\dbfam
\textfont\dbfam=\tendb \scriptfont\dbfam=\sevendb
\scriptscriptfont\dbfam=\fivedb
\def\db{\fam\dbfam\tendb}
\font\eufm=eufm10\font\eufms=eufm10\font\eufmss=eufm10\newfam\eufam
\textfont\eufam=\eufm\scriptfont\eufam=\eufms\scriptscriptfont\eufam=\eufmss

\font\tendbb=msbm10 at 12pt \font\sevendbb=msbm7 at 9pt
\font\fivedbb=msbm5 at 6pt
\newfam\dbbfam
\textfont\dbbfam=\tendbb \scriptfont\dbbfam=\sevendbb
\scriptscriptfont\dbbfam=\fivedbb

 \def \Z {{\db Z}}

\font\tenMmm=eusm10 at 12pt
\newfam\Mmmfam
\textfont\Mmmfam=\tenMmm

\def\illu #1 by #2 (#3){
  \vbox to #2{
    \hrule width #1 height 0pt depth 0pt
    \vfill
    \special{illustration #3} % this is the low-level interface
    }
  }

\begin{document}

\title[Graph polynomials and  symmetries]
{Graph polynomials and  symmetries}

%\keywords{Automorphism group, adjacency matrix, characteristic  polynomial, Tutte polynomial.}

\author{Nafaa Chbili}
\address{Department of Mathematical Sciences\\ College of Science UAE University \\
17551 Al Ain, U.A.E.} \email{nafaachbili@uaeu.ac.ae}
\urladdr{http://faculty.uaeu.ac.ae/nafaachbili}

\date{}

\begin{abstract}
In a recent paper, we studied the interaction between the automorphism group of a graph and its Tutte polynomial.  More precisely, we proved that   certain symmetries of graphs are clearly reflected by their Tutte polynomials.  The purpose of this paper is to  extend this study to other  graph polynomials. In particular, we prove that if a graph $G$ has a symmetry of prime order $p$, then its characteristic polynomial, with coefficients in the finite filed  $\mathbb{F}_p$,
 is determined  by the characteristic polynomial   of its quotient graph $\overline G$.  Similar  results are also proved  for some generalization of the  Tutte  polynomial.
\end{abstract}
\maketitle
\section{Introduction}
Let $V=\{v_1,\dots,v_n\}$ be a finite set. A weight on $V$ is a symmetric  function $w:V\times V \longrightarrow \Z_{+}\cup\{0\}$. The pair $G=(V,w)$ is called a weighted graph.
 In other words,  a weighted graph  can be seen as a finite graph possibly with multiple
 edges and loops. In particular, if  $G$ is a simple graph in the standard terminology,  then the corresponding  symmetric function $w$ is  defined by  $w(u,v)=1$ if $u$ and $v$ are adjacent and 0 otherwise.
 Throughout this paper, wherever no confusion may arise,  we will simplify notations   and  use the term graph
 instead of weighted graph.  If the vertices of  $G$ are ordered   $v_1, v_2, \dots,  v_n$, then the   adjacency matrix of $G$,  is the  $n$-square matrix  $A(G)=[w_{i,j}]$ where  $w_{i,j}=w(v_i,v_j)$.
 In general, since  the function $w$ is symmetric, then so is the matrix  $A(G)$.
  The characteristic polynomial of $A(G)$ defined as  $det(xI_n-A(G))$ is known to be  independent  from  the order  of the vertices. It is denoted hereafter by $\varphi_{G}(x)$.
 If $D(G)=[d_{i,j}]$ is the degree matrix of $G$, which is the diagonal $n$-square matrix
defined by $d_{ii}=deg(v_i)=w(v_i,v_i)$, then the Laplacian matrix of $G$  is defined by $L(G)=D(G)-A(G)$. The characteristic polynomial of $L(G)$ is also  independent of the ordering  of the vertices
and is denoted hereafter by $\psi_G(x)$. Both polynomials $\varphi_G(x)$ and $\psi_G(x)$ are well-known graph invariants and the study of their spectra is  of central importance  in graph theory.\\
 An automorphism of a  weighted graph $G$ is a permutation  $\sigma$ of the set of vertices  $V$ such that $w(\sigma(u),\sigma(v))=w(u,v)$ for any pair of vertices $u$ and $v$.
 The set of all automorphisms of $G$ forms a group denoted here  by  $Aut(G)$.
Let $p\geq 2$ be  an integer. A graph  $G$ is said to be  $p$-periodic  if its automorphism group  $Aut(G)$ contains an element $h$ of order $p$.
For our purpose, we assume that the action defined by the element $h$ has no fixed vertices, which means that $h^p=Id$ and $h^i(v)\neq v$ for all $1\leq i \leq p-1$.\\
 Given a $p$-periodic  weighted graph $G$,  we define its quotient graph $\overline G$ as the weighted graph $(\overline V, \overline w)$, where $\overline V$ is the quotient set of $V$ under the action of $h$ and $\overline w$ is the weight function on $\overline V$ defined by $\overline w (\overline u, \overline v)=\displaystyle\sum_{i=0}^{p-1}w(u,h^i(v))$. A 3-periodic weighted  graph and its quotient are  pictured   below.\\
 \vspace{2mm}
 \begin{center}
\includegraphics[width=8cm,height=3.5cm]{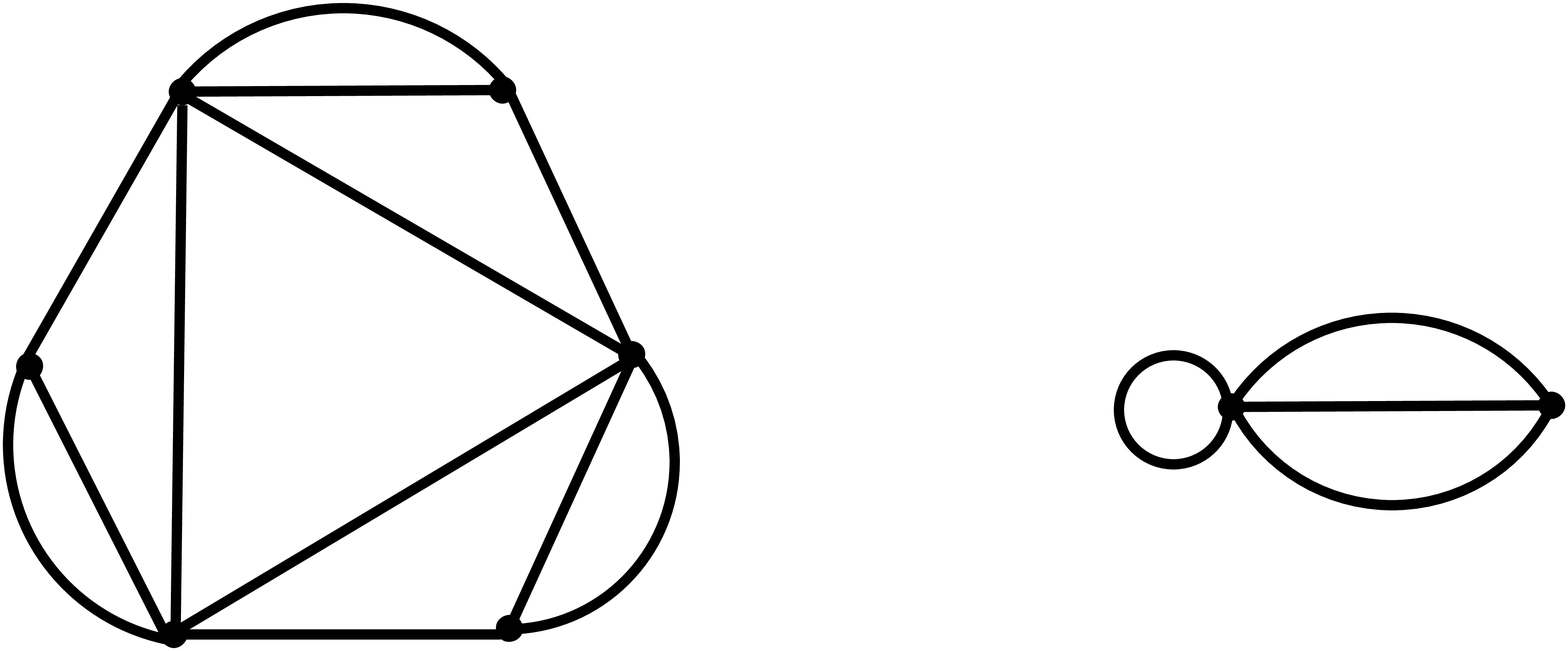}\\
Figure 1. A 3-periodic  graph (left) and its quotient graph (right)
\end{center}

In  \cite{Ch},  we studied the Tutte polynomial of  $p$-periodic  graphs and proved that certain  symmetries of  a given  graph are reflected by  its Tutte polynomial in a clear way.
 More precisely, we proved that if $G$  is $p$-periodic, with $p$ an odd prime, then     certain  coefficients of this two-variable  polynomial  are null modulo $p$.
It is worth mentioning that the original motivation behind our interest in this question  is the study of the quantum  invariants of periodic knots.
Recall  that the Jones polynomial of an alternating link is given by the Tutte polynomial of its Tait graph,  \cite{Th}. The Tutte polynomial is also related to  certain specialization of the  HOMFLYPT polynomial as it was proved in \cite{Ja}. The latter is known to be a good witness of  link symmetries, see  \cite{Ch1,Ch2,Pr,Tr1,Tr2,Yo}.
 The purpose of this paper is to  investigate whether the  way the Tutte polynomial interacts with graph symmetries extends to other graph polynomials. We study
 the behavior of characteristic polynomials $\varphi$ and $\psi$ defined above. In addition, we consider
the multi-variable  graph polynomial $U_G$  introduced in \cite{NW}. This graph invariant is an interesting  generalization of the Tutte polynomial. And the natural question we address
 here  is how does the condition satisfied by Tutte polynomial of a periodic graphs extend to the invariant  $U_G$.\\
Before stating our first theorem, we would like to mention  that several  results about the characteristic polynomial of  graphs with certain symmetries have been obtained earlier, see \cite{FKL, Wa} for instance. In this paper, we look at these polynomials as with coefficients in the finite field of $p$ elements $\mathbb{F}_p$. We shall prove the following necessary condition for a graph to be $p$-periodic.

\begin{thm}{\sl  Let $p$ be an odd prime and  $G$ be a $p$-periodic graph with quotient $\overline G$. Then $$\varphi_G(x)=\varphi_{\overline G}(x^p) \in \mathbb{F}_p[x].$$}
\end{thm}
\begin{exmp}
The Theorem above shows that the coefficient of $x^{k}$ in  $\varphi_G(x)$ is  null modulo $p$ whenever  $k$ is not a multiple of $p$.
For the 3-periodic  graph in Figure 1, we have $\varphi_G(x)=x^6 - 18 x^4 - 14 x^3 + 66 x^2 + 36 x-81$. Its of the quotient graph is $\varphi_{\overline G}(x)=x^2-2x-9$. The necessary condition given by Theorem 1.1 is clearly satisfied.\\
However, for the Frucht graph $F$ pictured below,  the characteristic polynomial  is given by
$$\varphi_F(x)=x^{12}-18 x^{10}-6x^9+115x^{8}+66x^{7}-309x^{6}-226x^5+309x^4+244x^3-68x^2-48x.$$
 By considering this polynomial with coefficients reduced modulo $p$, we can easily see that the graph  is not $p$-periodic for any $p$ odd prime.
\begin{center}
\includegraphics[width=3cm,height=3cm]{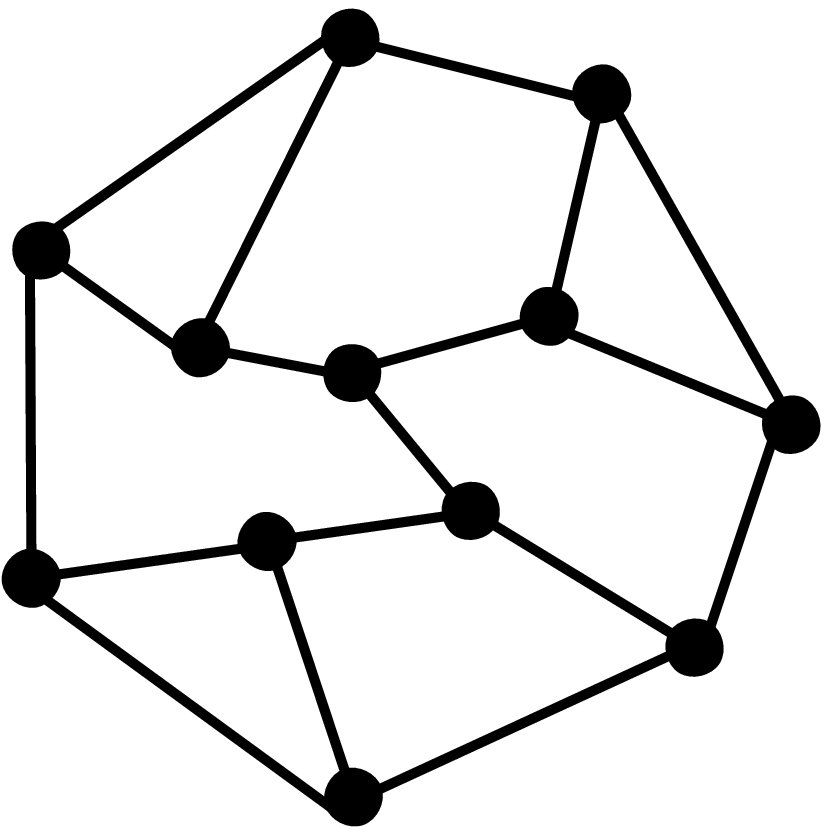}
\end{center}
\end{exmp}

A similar necessary condition for a graph to be periodic can be obtained in terms of the  characteristic polynomial of the  Laplacian matrix.
\begin{thm}{\sl { Let $p$ be an odd prime and  $G$ be a $p$-periodic graph. Then $\psi_G(x)$  is divisible by $x^p$ in  $\mathbb{F}_p[x^p]$}}.
\end{thm}
Notice that for  3-periodic graph in Figure 1,  we have  $\psi_G(x)=x^6 + 24 x^5 + 219 x^4 + 936 x^3 + 1845 x^2 + 1350 x$, which is congruent to $x^6$ modulo $p$. Hence, $\psi_G(x)$ is divisible by $x^3$.
\section{The Characteristic polynomial}
In this section, we will prove the congruence relation given in Theorem 1.1.  Similar arguments will be used to prove Theorem 1.3.  Let $G$ be a $p$-periodic graph. Then the finite cyclic
 group of order $p$ acts freely on the set of vertices of $G$. This set splits into a partition made up of $s$ orbits, where $s$ is the number of vertices of the quotient graph $\overline G$.
 Since the action is free and $p$ is prime, then each orbit is made up of exactly $p$ elements.  Let us label the vertices of $\overline G$ as $v_1, \dots, v_{s}$. Let $\pi$ be the canonical  surjection from
 $V$ to $\overline V$. We label the elements  of $\pi^{-1}(v_i)$ as $v_i^0, \dots, v_i^{p-1}$, so that $h(v_i^k)=v_i^{k+1}$ if $0 \leq k< p-1 $ and $h(v_i^{p-1})=v_i^{0}$.
The order of the vertices of $\overline G$ extends to a natural order of the vertices of $G$. Namely,  $v_1^{0}, v_{2}^{0},\dots, v_{s}^{0},v_{1}^{1},v_{2}^{1},\dots ,v_{s}^{1},v_{s}^{2},\dots ,v_{s}^{p-1}$. With respect to this order, the adjacency matrix of $G$ will be a block circulant matrix of the form:
 $$A(G)=\left [\begin{array}{ccccc}
A_0&A_1&A_2&\dots&A_{p-1}\\
A_{p-1}&A_0&A_1&\dots&A_{p-2}\\
\vdots&\vdots&\vdots&\vdots&\vdots\\

A_{1}&A_2&A_3&\dots&A_0
\end{array} \right ]$$

Where $A_j$ is the $s$-square  matrix whose $(i,k)$ entry is   $w(v_i^0,v_k^j)$. We write $A(G)=(A_0,A_1,\dots,A_{p-1})$. Notice that $A(G)$ is a block circulant matrix,   because  for any $m$, we have $w(h^m(v_i^0),h^m(v_k^j))=w(v_i^0,v_k^j)$. Let us illustrate this
by considering  the adjacency  matrix of the 3-periodic  graph pictured  in Figure 1. For this graph $A(G)=(A_0,A_1,A_2)$, where $A_0=\left [\begin{array}{cc}
0&2\\
2&0
\end{array} \right ]$,
$A_1=\left [\begin{array}{cc}
0&0\\
1&1
\end{array} \right ]$ and
$A_2=\left [\begin{array}{cc}
0&1\\
0&1
\end{array} \right ]$, which can be displayed as:

$$\left [\begin{array}{cc|cc|cc}
0&2&0&0&0&1\\
2&0&1&1&0&1\\
\hline
0&1&0&2&0&0\\
0&1&2&0&1&1\\
\hline
0&0&0&1&0&2\\
1&1&0&1&2&0
\end{array} \right ]$$

Now, let $\zeta$ be a root of unity of order $p$. For  $0 \leq k \leq p-1$, we define  the $s$-square   matrix  $T_k=\displaystyle\sum_{j=0}^{p-1}\zeta^{kj}A_j$. According to \cite{Fr}, the characteristic polynomial of $A(G)$ is given by
$\varphi_G(x)=\displaystyle \prod_{k=0}^{p-1} \varphi_{T_k}(x)$, where  $\varphi _{T_k}(x)$ denotes the characteristic polynomial of the matrix $T_k$. It is easy to see that $T_0$ is the adjacency matrix of the quotient graph $\overline G$, hence $\varphi_{T_0}(x)=\varphi_{\overline G}(x)$.
% which leads to the identity
%$\varphi_G(x)=\Pi_{k=0}^{p-1} \varphi_{\overline G}(\zeta^kx)$.\\

Notice that each term in the product formula  is a polynomial on $x$ with coefficients in $\Z[\zeta]$ while $\varphi_G(x)$ belongs to $\Z[x]$ since the incidence matrix $A(G)$ is symmetric with integral coefficients.
Consider the   following homomorphism:
$$\begin{array}{rlll}
f_p:&\Z[\zeta]& \longrightarrow& \Z_p\\
&\sum n_i\zeta^i &\longmapsto& \sum n_i
\end{array}$$

This map  is well defined since the sum  $\sum n_i$ is considered modulo $p$. Moreover,
it extends to a homomorphism between  the polynomial rings $f_p: \Z[\zeta] [x]\longrightarrow \Z_p[x]$. We have clearly
 $$f_p(\varphi_{T_k}(x))=f_p(\varphi_{T_0}(x))=\varphi_{\overline G}(x).$$
 Consequently, we obtain the following congruence modulo $p$
$$\varphi_G(x)\equiv[\varphi_{\overline G}(x)]^p\equiv\varphi_{\overline G}(x^p).$$
This proofs Theorem 1.1.\\
The arguments used above apply to the Laplacian matrix $L(G)$ as well. If $G$ is $p$-periodic graph, then we can order the vertices as above  so that   $L(G)$ is a
 block circulant matrix of the form $(B_0,B_1,\dots,B_{p-1})$.
For $0\leq k \leq p-1$, consider the $s$-square matrix  $R_k=\displaystyle\sum_{j=0}^{p-1}\zeta^{kj}B_j$.
The  characteristic polynomial of the  block circulant matrix $L(G)$,  is given by the following identity
   $$\psi_G(x)=\displaystyle\prod_{k=0}^{p-1}\varphi_{R_k}(x)$$
where $\varphi_{R_k}(x)$ denotes the characteristic polynomial of the matrix $R_k$.
Using the map $f_p$ defined above we get the following congruence modulo $p$,
$$\psi_G(x)\equiv \displaystyle\prod_{i=1}^{p}\varphi_{R_0}(x)\equiv (\varphi_{R_0}(x))^p.$$

Since $R_0=B_0+\dots+B_{p-1}$, then the sum the of entries of each row of $R_0$ is zero. Hence, 0 is an  eigenvalue of $R_0$ and its characteristic polynomial is    divisible by $x$. Consequently, the polynomial
 $\psi_G(x)$ is divisible by $x^p$ in $\mathbb{F}_p[x]$.\\
%%%%%%%%%%%%%%%%%%%%%%%%%%%%%%%%%%%%%%%%%%%%%%%%%%%%%%%%%%%%%%%%%%%%%%%%%%%%%%%%%%%%%%%%%%%%%%%%%%%%%%%%%%%%%%%%%%%%%%%%%%%%%%%%%%
\section{The  weighted graph polynomial}
%%%%%%%%%%%%%%%%%%%%%%%%%%%%%%%%%%%%%%%%%%%%%%%%%%%%%%%%%%%%%%%%%%%%%%%%%%%%%%%%%%%%%%%%%%%%%%%%%%%%%%%%%%%%%%%%%%%%%%%%%%%%%%%%%%%%%%

The Tutte polynomial is an isomorphism invariant  of graphs that have been introduced in \cite{Tu}. It is a two-variable polynomial  $\tau_G(x,y)$ with integral coefficients  that generalizes the well known chromatic polynomial of graphs.
The Tutte polynomial has many  interesting applications in different fields such as knot theory and statistical physics.
This invariant carries important information about the graph and it can be defined in several  ways.  The most simple one is recursive   and based on  the deletion-contraction formula. Let $G$ be a graph and $e$ one of its edges. We denote by $G-e$ the graph obtained from $G$ by deleting edge $e$. By $G/e$ we denote the graph obtained  by identifying the two endpoints of  $e$. The Tutte polynomial can be defined by the following relations
$$ \tau_G(x,y) = \left\{
\begin{array}{llll}
x\tau_{G/e}(x,y)&&& \mbox { if $e$ is a bridge} \\
 y\tau_{G-e}(x,y) &&& \mbox {if $e$ is a loop}\\
\tau_{G-e}(x,y)+\tau_{G/e}(x,y) &&& \mbox{if $e$ is an ordinary edge}
\end{array} \right.
 $$

and the initialization $\tau_{E_n}(x,y)=1$, where $E_n$ is the graph with $n$ vertices and no edges. \\
In \cite{Ch}, we considered the version of the Tutte polynomial defined by setting  $T_G(s,t)=\tau_G(s+1,t+1)$ and
 proved that the Tutte polynomial strongly interacts with the automorphism group of the graph.  Actually, we proved that if a   graph $G$ admits an action of the finite cyclic group $\Z_p$ which leaves no edge fixed then
 $T_G(s,t) = \displaystyle\sum_{i,j}a_{i,j}s^it^j$ where $a_{i,j}\equiv 0$  modulo $p$  whenever $i-j$ is not congruent to $n-1$  modulo $p$, where $n$ is the number of vertices of the graph.\\
An interesting generalization of the Tutte polynomial has been introduced in \cite{NW}, where the authors define an multi-variable polynomial of vertex-weighted graphs which specializes to the Tutte polynomial.
Here, we will  only consider the restriction of that
 invariant to  ordinary graphs as defined in Section 5 of  \cite{NW}. Let $G$ be a graph with vertex set $V$ and edge set $E$. Let  $y,x_1,\dots x_n$, be commuting indeterminates.
 Then, $U_G(x_1\dots x_n,y)$ denoted for short $U_G(\boldsymbol {x},y)$ is the invariant of graphs defined by the
 formula:
 $$U_G(\boldsymbol {x},y)=\displaystyle\sum_{A\subseteq E} x_{n_1}\dots x_{n_k}(y-1)^{|A|-r(A)}.$$
Where the sum is taken through  all subsets of the edge set $E$ of $G$, $n_1,\dots, n_k$ are the number of vertices of the different components of $G|A$. Here $G|A$ denotes the restriction of $G$ to $A$, which is obtained from $G$ by removing all the edges not in $A$.   Finally, $r(A)$ denotes the rank of $A$ defined as  $r(A)=n-k(G | A)$, where
$k(G|A)$ is the number of connected components of $G|A$. The Tutte polynomial is a specialization  of  $U_G$ which is given by the following idendity
$$T_G(x,y)=(x-1)^{-k(G)}U_G(x_i=x-1,y).$$
It is worth mentioning that $U_G$ is much stronger invariant than the Tutte polynomial as, in many cases,  it distinguishes between  pairs of graphs which share the same Tutte  polynomial \cite{NW}.
\begin{thm} Let $p$ be an odd prime,  $G$ a $p$-periodic graph on $n$ vertices and $\overline G$ its quotient graph.  Then, the following congruence holds
$$ U_G({\boldsymbol {x}},y)\equiv U_{\overline G}({\boldsymbol {\overline x}},y)$$ modulo $p$, $x_k^p-x_k$, $x_{pk}-x_k$ and $y^p-y$, for all $k$, where $U_{\overline G}({\boldsymbol {\overline x}},y)$ denotes
 the polynomial $U_{\overline G}(x_1\dots x_{\frac{n}{p}},y)$.
\end{thm}
\begin{proof} Notice that the finite group action on $G$ defines an action on  the subsets of $E$. Moreover, a subset $A$  of $E$ is either invariant by the action or belongs to an orbit made up of $p$ elements. It is clear that
if the orbit of subset $A$ is made up of $p$ elements, then the contributions of these elements in the  formula defining $U_G(\boldsymbol {x},y)$ add to zero modulo $p$.
Hence, to compute  $U_G(\boldsymbol {x},y)$ modulo $p$, one can only consider subsets $A$ of $E$ which are invariant by the action that $E$ inherits from the action  of $\Z_p$ on $G$.  If $A \subseteq E$ is invariant by the action, then we denote by $\overline A$, the quotient of $A$, which is actually a subset of  $\overline E$  the set of edges of the quotient graph $\overline G$.  Obviously, there is a one-to-one correspondence between the subsets of $E$ invariant by the action and the subsets of $\overline E$.
Assume that  $\overline G |\overline A$ is connected having  $k$ vertices. Then  the contribution of $\overline A$ in $U_{\overline G}$ is $x_k(y-1)^{|\overline A|-r(\overline A)}$. However, the contribution of the subset $A$  in $U_G$ is  $x_{pk}(y-1)^{|A|-r(A)}$  if $G|A$ is connected.
In the case  $G|A$ has $p$ connected components, then each component has $k$ vertices and  the contribution of $A$ in the summation formula is $x_k^p(y-1)^{|A|-r(A)}$. Notice that in both cases, the two  terms are congruent modulo  $x_k^p-x_k$, $x_{pk}-x_k$ and $y^p-y$.
If $\overline G |\overline A$  is not connected, then to each component of $\overline G |\overline A$ corresponds either one or $p$ connected components of $G|A$ as it was explained in the previous case. Consequently, the contribution of $\overline A$ in $U_{\overline G}$ is the same as the contribution of $A$ in $U_G$ modulo the ideal generated by $x_k^p-x_k$, $x_{pk}-x_k$ and $y^p-y$. This completes the proof.
\end{proof}
If we restrict the condition given in Theorem 3.1  to the case of the Tutte polynomial, then we obtain the following congruences which were  introduced earlier in  \cite{Ch}.
\begin{coro}
 Let $p$ be an odd prime,  $G$ a $p$-periodic   connected graph on $n$ vertices  and $\overline G$ its quotient. Then,
\begin{enumerate}
\item $T_G(s,t)\equiv T_{\overline G}(s,t)$ modulo $p, s^p-s$ and $t^p-t$.
\item $T_G(s,t)=\displaystyle\sum a_{i,j}s^it^j$, where  $a_{i,j}$ is null modulo $p$,  whenever $i-j$ is not congruent to $n-1$  modulo $p$.
\end{enumerate}
\end{coro}
\begin{proof} The proof of   condition (1) is straightforward by setting $x_i=s$, $y-1=t$ in the congruence given by Theorem  3.1. The second condition is a consequence of the fact that only periodic sets $A$ are to be considered when we compute $U_G$ modulo $p$. If $G|A$ has $k$ connected components then the term that corresponds to $A$ in the sum formula becomes, after setting $x=s$ and $y-1=t$,  $s^{k}t^{|A|-r(A)}=s^{k}t^{|A|-(n-k)}$. Since $|A|$ is a multiple of $p$, then the difference between the two powers of $s$ and $t$ is congruent to $n$ modulo $p$. We conclude using  the fact that if $G$  is connected then  $T_G(s,t)=s^{-1}U_G(x_i=s,y=t+1)$.
\end{proof}
%%%%%%%%%%%%%%%%%%%%%%%%%%%%%%%%%%%%%%%%%%%%%%%%%%%%%%%%%%%%%%%%%%%%%%%%%%%%%%%%%%%%%%%%%%%%%%%%%%%%%%%%%%%%%%%%%%%%%%%%%%%%%%%%%%%%%%%%%

\end{document}